\documentclass[12pt]{amsart}
\usepackage{amssymb}
\usepackage{mathrsfs} 
 
\usepackage[all]{xy}
\usepackage{textcomp}

\makeatletter 
\@addtoreset{equation}{section} 
\makeatother

\input cyracc.def
\font\tencyr=wncyr10
\def\cyr{\tencyr\cyracc}

\newcommand{\CC}{\mathbb{C}}
\newcommand{\PP}{\mathbb{P}}
\newcommand{\FF}{\mathbb{F}}
\newcommand{\ZZ}{\mathbb{Z}}

\newcommand{\Gr}{\operatorname{Gr}}

\newcommand{\Pic}{\operatorname{Pic}}

\renewcommand\labelenumi{(\roman{enumi})}
\renewcommand\theenumi{(\roman{enumi})}

\newtheorem{theorem}{Theorem}[section]
\newtheorem{proposition}[theorem]{Proposition}
\newtheorem{lemma}[theorem]{Lemma}
\newtheorem*{lemma*}{Lemma}
\newtheorem{corollary}[theorem]{Corollary}

\newtheorem{claim}[theorem]{Claim}
\newtheorem*{claim*}{Claim}

\theoremstyle{definition}

\newtheorem{remark}[theorem]{Remark}
\newtheorem{construction}[theorem]{Construction}
\newtheorem{sit}[theorem]{}

\theoremstyle{remark}
\newtheorem{convention}[theorem]{Convention}

\newcommand{\rk}{\operatorname{rk}}

\newcommand{\Sing}{\operatorname{Sing}}

\newcommand{\Aut}{\operatorname{Aut}}

\author{Yuri Prokhorov}
\thanks{The first author was partially supported by the grants 
RFBR 
\textnumero 
15-01-02164, 15-01-02158, 15-51-50045{{\cyr{YAF\_a}}}, 
and a subsidy granted to the HSE by the Government of the Russian Federation
for the implementation of the Global Competitiveness Program.
Both authors thank the Max Planck Institute of Mathematics in Bonn,
where a part of the paper was written, for excellent working conditions and a generous support. }

\address{Yuri Prokhorov:
Steklov Mathematical Institute of Russian Academy of Sciences
\newline\indent
Department of Algebra, 
Moscow State Lomonosov University
\newline\indent
National Research University
Higher School of Economics (Russian Federation)
}
\email{prokhoro@gmail.com}

\author{Mikhail Zaidenberg}
\address{Mikhail Zaidenberg: Universit\'e
Grenoble I, Institut Fourier, UMR 5582 CNRS-UJF, BP~74, 38402 Saint
Martin d'H\`eres cedex, France} \email{zaidenbe@ujf-grenoble.fr }

\title{New examples of cylindrical Fano fourfolds}

\begin{document}
\begin{abstract}
We construct new families of smooth Fano fourfolds with Picard rank 1,
which contain cylinders, i.e., Zariski open subsets of form $Z\times{\mathbb A}^1$,
where $Z$ is a quasiprojective variety.
The affine cones over such a fourfold admit effective $\mathbb{G}_{\operatorname{a}}$-actions.
Similar constructions of cylindrical Fano threefolds and fourfolds were done previously in
\cite{Kishimoto-Prokhorov-Zaidenberg, Kishimoto-Prokhorov-Zaidenberg-Fano, Prokhorov-Zaidenberg-2014}.
\end{abstract}
%
%
\maketitle
\section{Introduction} All varieties in this paper are algebraic and are defined over $\CC$.
A smooth projective variety $V$
is called {\em cylindrical} if it contains a cylinder, i.e.,
a principal Zariski open subset $U$ isomorphic to a product $Z\times{\mathbb A}^1$,
where $Z$ is a variety and ${\mathbb A}^1$ stands for the affine line
(\cite{Kishimoto-Prokhorov-Zaidenberg,
Kishimoto-Prokhorov-Zaidenberg-criterion}).

Assuming that $ \rk \Pic(V)=1$,
by a criterion of
\cite[Cor.\ 3.2]{Kishimoto-Prokhorov-Zaidenberg-criterion},
the affine cone over $V$ admits an effective action of the additive group
$\mathbb{G}_{\operatorname{a}}$
if and only if $V$ is cylindrical. In the latter case $V$ is a Fano variety. Indeed, the existence of a 
$\mathbb{G}_{\operatorname{a}}$-action on the affine cone over
$V$ implies that $V$ is uniruled. Since $ \rk \Pic(V)=1$, $V$ must be a Fano variety. 
In \cite{Kishimoto-Prokhorov-Zaidenberg, Kishimoto-Prokhorov-Zaidenberg-Fano, Prokhorov-Zaidenberg-2014} 
several families of smooth cylindrical Fano threefolds and fourfolds with Picard number 1 were constructed.
Here we provide further examples of such fourfolds. Let us recall the standard terminology and notation.

\subsection{Notation}
Given a smooth Fano fourfolds $V$ with Picard rank 1, the \textit{index} of $V$
is the integer $r$ such that $-K_V=rH$, where $H$ is the ample divisor generating
the Picard group: $\Pic (V)=\ZZ\cdot H$ (by abuse of notation, we denote by the same letter a divisor 
and its class in the Picard group). 
The \textit{degree} $d=\deg V$ is the degree with respect to $H$.
It is known that $1\le r\le 5$. Moreover,
if $r=5$ then
$V\cong \PP^4$, and if $r=4$ then $V$ is a quadric in
$\PP^5$. Smooth Fano fourfolds of index $r=3$ are called {\em del
Pezzo fourfolds};
their degrees vary in the range $1\le d\le 5$ (\cite{Fujita1980}-\cite{Fujita-620281}).
Smooth Fano fourfolds of index $r=2$ are called {\em Mukai fourfolds};
their degrees are even and can be written as $d=2g-2$, where $g$ is called the \textit{genus} of $V$.
The genera of Mukai fourfolds satisfy
$2\le g\le 10$ (\cite{Mukai-1989}). The classification of Fano fourfolds of
index $r=1$ is not known.

According to \cite[Thm.\ 0.1]{Prokhorov-Zaidenberg-2014} a smooth intersection of two quadrics in $\PP^6$ 
is a cylindrical del Pezzo fourfold of degree $4$.
A smooth del Pezzo fourfold $W=W_5\subset \PP^7$ of degree $5$ is also cylindrical ({\it ibid.}).

\subsection{On the content.} Starting with the del Pezzo quintic fourfold $W$ and performing suitable Sarkisov links 
we constructed in \cite{Prokhorov-Zaidenberg-2014}
two families of cylindrical Mukai fourfolds $V_{12}$ of genus $7$ and $V_{14}$ of genus $8$. Proceeding in a similar 
fashion, in the present paper
we construct two more families of cylindrical Mukai fourfolds $V_{16}$ of genus $9$ and $V_{18}$ of genus $10$, see 
Theorem \ref{theorem-1} and Corollary \ref{cor:cylinder}. These are the main results of the paper.

The paper is divided into 6 sections. After formulating in Section \ref{sec:Sarkisov} our principal results, we give 
in Section \ref{sec:preliminaries} necessary preliminaries. In particular, we recall some useful facts from 
\cite{Prokhorov-Zaidenberg-2014}. In Section \ref{sec:proof-of-theorem-1} we prove Theorem \ref{theorem-1}
about the existence of suitable Sarkisov links. This theorem depends on the existence of certain specific 
surfaces in the quintic fourfold $W$. Section \ref{sec:constructions} is devoted to constructions of such 
surfaces, see Proposition \ref{lemma-existence-surfaces}. The resulting Mukai fourfolds $V_{16}$ and $V_{18}$ 
occur to be cylindrical, with a cylinder coming from a one on $W$ via the corresponding Sarkisov link, see 
Corollary \ref{cor:cylinder}. Section 6 contains concluding remarks and some open problems. 

\section{Main results} \label{sec:Sarkisov}
The following theorem describes the Sarkisov links used in our constructions.

\begin{theorem}
\label{theorem-1}
Let $W=W_5\subset \PP^7$ be a del Pezzo fourfold of degree $5$, and let
$F\subset W\cap \PP^6$ be a smooth surface of one of the following types:
\begin{enumerate}
\renewcommand\labelenumi{\alph{enumi})}
\renewcommand\theenumi{\textup{\alph{enumi})}}
\item\label{case-g=10}
$F\subset \PP^6$ is a rational normal quintic scroll, $F\cong \FF_1$,
and

\item\label{case-g=9}
$F\subset \PP^6$ is an anticanonically embedded sextic del Pezzo surface such that $c_2(W)\cdot F=26$
\textup(see Lemma \textup{\ref{lem:enumerative}}\textup).
\end{enumerate}
Suppose that $F$ does not intersect any plane in $W$ along a \textup(possibly, degenerate\textup) conic.
Then there is a commutative diagram
\begin{equation}\label{eq:diagram}
\vcenter{
\xymatrix{
&D\ar[dl]&\hookrightarrow&\widetilde W\ar[dr]^{\varphi}\ar[dl]_{\rho}&\hookleftarrow& \tilde E\ar[dr]
\\
F&\hookrightarrow&W\ar@{-->}[rr]^{\phi} &&V&\hookleftarrow& S
}
}
\end{equation}
where
\begin{itemize}
\item
$V=V_{2g-2}\subset \PP^{g+2}$ is a Mukai fourfold of genus $g=10$ in case \ref{case-g=10}
and
$g=9$ in case \ref{case-g=9};
\item
the map $\phi: W\dashrightarrow V\subset \PP^{g+2}$ is given by
the linear system of quadrics passing through $F$, while $\phi^{-1}: V \dashrightarrow W$ is the 
projection from the linear span $\langle S\rangle$ of $S$. 
\end{itemize}
Furthermore,
\begin{enumerate}
\item
$\rho: \widetilde W\longrightarrow W$ is the blowup of $F$
with exceptional divisor $D$, and
$\varphi: \widetilde W\longrightarrow V$ is
the blowup of a smooth surface $S\subset V$ with exceptional divisor
$\tilde E$, where
\begin{itemize}
\item
in case \ref{case-g=10}
$S\subset \PP^4\subset \PP^{12}$ is a normal cubic scroll with
$c_2(V)\cdot S=7$, and
\item
in case \ref{case-g=9}
$S\subset \PP^3\subset \PP^{11}$ is a quadric with
$c_2(V)\cdot S=5$;
\end{itemize}
\item
if $H$ is an ample generator of $\Pic(W)$
and $L$ is an ample generator of $\Pic(V)$, then on $\widetilde W$ we have
\begin{equation}\label{equation-intersections}
\begin{array}{ll}
\rho^*H\equiv\varphi^* L-\tilde E,&
D\equiv \varphi^* L-2\tilde E,
\\[7pt]
\varphi^* L\equiv 2\rho^*H-D,&
\tilde E\equiv \rho^*H-D.
\end{array}
\end{equation}
\end{enumerate}
\end{theorem}

The proof is done in Section \ref{sec:proof-of-theorem-1}. In Section \ref{sec:constructions} 
we establish the existence 
of surfaces $F$ as in Theorem \ref{theorem-1}. 

Using this theorem, we deduce our main results.

\begin{theorem}\label{main-theorem}
Under the assumptions of Theorem \textup{\ref{theorem-1}},
there is an isomorphism 
\begin{equation*}
V\setminus\varphi(D)\cong W\setminus\rho(\tilde E)\,,
\end{equation*}
where $\varphi(D)$ is a hyperplane section of $V=V_{2g-2}\subset \PP^{g+2}$ singular along 
$S=\varphi(\tilde E)$, and 
$\rho(\tilde E)=W\cap \langle F\rangle$ is a singular hyperplane section of
$W=W_5\subset \PP^7$ by the linear span of $F$.
\end{theorem}

Recall the following fact (\cite[Thm.\ 4.1]{Prokhorov-Zaidenberg-2014}).

\begin{theorem}\label{thm:W-cyl}
For any hyperplane section $M$ of $W$, the complement $W\setminus M$ contains a cylinder.
\end{theorem}

\begin{corollary}\label{cor:cylinder} 
Any Fano fourfold $V$ as in Theorem \textup{\ref{theorem-1}} is cylindrical. 
\end{corollary}

\begin{proof} Since $M=\rho(\tilde E)$ is a hyperplane section of $W$, the complement
$W\setminus \rho(\tilde E)$ contains a cylinder. Hence also
$V\setminus\varphi(D)\cong W\setminus\rho(\tilde E)$ does, and so, 
$V$ is cylindrical.
\end{proof}

\section{Preliminaries}\label{sec:preliminaries} 
\begin{sit} 
Recall the following notation, see e.g. \cite[\S 3]{Prokhorov-Zaidenberg-2014}.
There are two types of planes in the Grassmannian
$\Gr(2, 5)$, namely, the Schubert varieties $\sigma_{3,1}$ and
$\sigma_{2,2}$ (\cite[Ch.\ 1, \S 5]{Griffiths-Harris-1978}), where
\begin{itemize}
\item
$\sigma_{3,1}=\{l\in \Gr(2, 5)\mid p\in l \subset h\}$ with $h\subset \PP^4$ a fixed
hyperplane and $p\in h$ a fixed point;
\item
$\sigma_{2,2}=\{l\in \Gr(2, 5)\mid l \subset e\}$ with $e\subset \PP^4$ a fixed
plane.
\end{itemize}
In the terminology of \cite[\S 10]{Fujita-620281}, the $\sigma_{3,1}$-planes (the $\sigma_{2,2}$-planes, 
respectively) are called planes
of {\em vertex type} (of {\em non-vertex type}, respectively).
\end{sit}

\begin{sit}
Let $W=W_5\subset \PP^7 $ be a del Pezzo fourfold of index $3$ and degree $5$. Due to \cite{Fujita-620281} 
such a variety is 
unique up to isomorphism and can be realized as
a section of $\Gr(2,5)\subset \PP^9$
by two general hyperplanes.
By the Lefschetz hyperplane section theorem $\Pic(W)\cong \ZZ$. We have $-K_W=3H$, where
$H$ is the ample generator of $\Pic(W)$.
The variety $W$ is an intersection of quadrics (see
\cite[Ch.\ 1, \S 5]{Griffiths-Harris-1978}).
\end{sit}

The following proposition proven in \cite{Todd1930} (see also \cite[Sect.\ 2]{Fujita-1986-1}) 
deals with the planes in the fourfold $W=W_5$.

\begin{proposition}
\label{Proposition-2.2.}
Let $W=W_5\subset \PP^7 $ be a Fano fourfold of index $3$ and degree $5$.
Then the following hold.
\begin{enumerate}
\item
$W$ contains a unique $\sigma_{2,2}$-plane $\Xi$, a one-parameter
family $(\Pi_t)$ of $\sigma_{3,1}$-planes, and no further plane.

\item
Any $\sigma_{3,1}$-plane $\Pi$ meets $\Xi$ along a tangent line to a fixed conic
$\delta\subset \Xi$.

\item
Any two $\sigma_{3,1}$-planes $\Pi'$ and $\Pi''$ meet at a point $p\subset \Xi \setminus \delta$.

\item
Let $R$ be the union of all $\sigma_{3,1}$-planes on $W$.
Then $R$ is a hyperplane section of $W$ and $\Sing R = \Xi$.

\item
There is a 1-parameter family of lines in $W$ through each point in $W$. A line $l\subset W$ meets the plane $\Xi$ if and 
only if $l\subset R$, and then $l$ is contained in a plane in $R$. 
\end{enumerate}
\end{proposition}

%
By abuse of notation, the cohomology class associated with 
an algebraic subvariety will be denoted by the same letter as the subvariety itself.
By the Lefschetz hyperplane section theorem, the group $H^4(W,\ZZ)$ is torsion free, since the group $H^4(\Gr(2,5),\ZZ)$ is.
In the next lemma we describe a natural basis in $H^4(W,\ZZ)$,
see \cite[Cor.\ 4.2 and 4.7]{Prokhorov-Zaidenberg-2014}. 

\begin{lemma}\label{lem:H4} The group $H^4(W,\ZZ)$
is freely generated by the classes of the planes $\Xi$ and $\Pi$, where
\begin{equation}\label{eq:ci}
\Pi^2 = 1,\quad \Xi^2 = 2, \quad \Pi\cdot\Xi = -1,\quad\mbox{and}\quad c_2(W)=9\Xi+13\Pi\,.
\end{equation}
\end{lemma}

\begin{lemma}\label{lem:enumerative}
\begin{enumerate}
\renewcommand\labelenumi{\alph{enumi})}
\renewcommand\theenumi{\textup{\labelenumi}}
\item\label{case-g=10-a}
Let $F\subset W\cap \PP^6$ be a smooth rational quintic scroll.
Then 
\begin{equation*}
F\equiv 2 \Xi+3\Pi\quad\mbox{and}
\quad F\cdot \Xi= 1,\,\, F\cdot\Pi= 1,\,\, c_2(W)\cdot F=22\,.
\end{equation*}

\item\label{case-g=9-a}
Let $F\subset W\cap \PP^6$ be a smooth anticanonically embedded sextic del Pezzo surface. Then
either
\vspace{10pt}
\begin{enumerate}
\itemsep=10pt
\renewcommand\labelenumii{\alph{enumi}\arabic{enumii})}
\renewcommand\theenumii{\textup{-\alph{enumi}\arabic{enumii})}}
\item\label{case-g=9-a-26}
$F\equiv 2 \Xi+4\Pi$ and
$F\cdot \Xi= 0$, $F\cdot\Pi= 2$, $c_2(W)\cdot F=26$, or
\item\label{case-g=9-a-27}
$F\equiv 3 \Xi+3\Pi$ and
$F\cdot \Xi= 3$, $F\cdot\Pi= 0$, $c_2(W)\cdot F=27$.
\end{enumerate}
\end{enumerate}
\end{lemma}

\begin{proof}
By Lemma \ref{lem:H4} one can write $F\equiv a \Xi+b\Pi$, where
\begin{equation}\label{eq:a-b}
a+b=\deg F,\qquad c_2(W)\cdot F=5a+4b.
\end{equation}
From the exact sequence
\begin{equation*}
0\longrightarrow \mathscr T_F \longrightarrow \mathscr T_W \longrightarrow \mathscr N_{F/W}
\longrightarrow 0\,
\end{equation*}
we deduce
\begin{equation}\label{eq:c1}
c_1(W)|_F= c_1(F)+ c_1( \mathscr N_{F/W})\,
\end{equation}
and
\begin{equation}
\begin{aligned}\label{eq:c2}
{\quad} c_2(W)\cdot F &= c_2(F)+ c_1(F)\cdot c_1( \mathscr N_{F/W}) +c_2(\mathscr N_{F/W})
\\[7pt]
&= c_2(F)- c_1(F)^2+ c_1(F)\cdot c_1(W)|_F +c_2(\mathscr N_{F/W})\,.\\
\end{aligned}
\end{equation}
The Noether formula for the rational surface $F$ can be written as follows:
\begin{equation*}
c_2(F)- c_1(F)^2=2c_2(F) -12\,.
\end{equation*}
Note that
\begin{equation*}
c_2(\mathscr N_{F/W})= F^2= 2a^2+b^2-2ab\,.
\end{equation*}
Since $c_1(W)=\mathcal{O}_W(3)$, from \eqref{eq:a-b} and \eqref{eq:c2} we obtain
\begin{multline*}
c_2(W)\cdot F=5a+4b=2c_2(F) -12 + c_1(F)\cdot \mathcal{O}_W(3)|_F + 2a^2+b^2-2ab\,.
\end{multline*}
In case \ref{case-g=10-a} using \eqref{eq:a-b} and the latter equality we get $b=5-a$ and
\begin{eqnarray*}
c_2(W)\cdot F&=&20+a=
5a^2 - 20a + 42\,,\quad\mbox{hence}\quad a=2\,.
\end{eqnarray*}
Similarly, in case \ref{case-g=9-a} we have $b=6-a$ and
\begin{eqnarray*}
c_2(W)\cdot F=24+a=5a^2-24a+54
\,,\quad\mbox{hence}\quad a\in\{ 2, 3\}\,.
\end{eqnarray*}
Now the assertions follow.
\end{proof}

\begin{remark}\label{lem:unique-hyperplane} For a surface
$F$ as in Lemma \ref{lem:enumerative} we have $\dim \langle F\rangle=6$. Hence $F$ is contained in a 
unique hyperplane section $\langle F\rangle\cap W\subset \PP^7$.
\end{remark}

\section{Construction of quintic and sextic surfaces $F\subset W$}\label{sec:constructions}
In this section we prove the existence of surfaces $F$ satisfying the assumptions of Theorem \ref{theorem-1}. 
Our main results can be stated as follows.

\begin{proposition}\label{lemma-existence-surfaces}
The quintic fourfold $W\subset\PP^7$ admits hyperplane sections which contain
\begin{enumerate}
\item[a)] 
a rational quintic scroll $F=F_5\subset \PP^6$,
\end{enumerate}
and other ones which contain
\begin{enumerate} 
\item[b)]
an anticanonically embedded sextic del Pezzo surface $F=F_6\subset \PP^6$
of type \ref{case-g=9-a-26}.
\end{enumerate}
In both cases, the surface $F$ can be chosen so that none of the planes in $W$ meets $F$ along a 
(possibly, degenerate) conic.
\end{proposition}

\begin{proof}[Proof of Proposition \textup{\ref{lemma-existence-surfaces}} {\rm b).}]

We start with a smooth sextic del Pezzo threefold $X=X_6\subset \PP^7$. Up to isomorphism, 
there is a unique such threefold $X$ with $\rk\Pic X=2$ (\cite{Fujita1980}, \cite{Iskovskikh-Prokhorov-1999}). 
In fact, the latter is the threefold which parametrizes the complete flags in $\PP^2$.
Consider the following diagram (\cite[\S 8]{Prokhorov-GFano-1}):
\begin{equation*}
\xymatrix{
&\tilde X\ar[dr]\ar[dl]&
\\
X\ar@{-->}[rr]&& U\subset \PP^{6}
}
\end{equation*}
where $U=U_5\subset \PP^{6}$ is a quintic del Pezzo threefold
with two nodes (ordinary double points), $X \dashrightarrow U=U_5\subset \PP^{6}$
is the projection from a general point $P\in X$, and $\tilde X\to X$ is the blowup of $P$.
Recall that $X$ can be realized as a smooth divisor
of bidegree $(1,1)$ in $\PP^2\times \PP^2$ (see, e.g., \cite{Fujita1980},
\cite{Iskovskikh-Prokhorov-1999}).
The natural projections ${\rm pr}_1,{\rm pr}_2: X\to \PP^2$ define
$\PP^1$-bundles with total space $X$.
Let $l_i$, $i=1,2$, be the corresponding fibers passing through $P$.
Then $l_1, l_2$ are contracted to the nodes $P_1,P_2\in U$.
The threefold $U$ contains a unique
plane $\mathcal{P}$, and this plane is
the image of the exceptional divisor of $\tilde X\to X$ (\cite[\S 8]{Prokhorov-GFano-1}).

The intersection $Z$ of $X$ with a general divisor
of bidegree $(1,1)$ in $\PP^2\times \PP^2$ is a smooth sextic del Pezzo
surface $Z\cong Z_6\subset \PP^6$.
We can choose $Z$ so that $P\not\in Z$. Let $F\subset U$ be the image of $Z$.
Then $F=F_6\subset W\cap \PP^6$ is an anticanonically embedded smooth sextic del Pezzo surface,
and $F\cap \mathcal{P}=\{P_1,P_2\}$.

Note that the del Pezzo quintic threefold $U=U_5\subset \PP^{6}$ with two nodes as above
is unique up to isomorphism.
On the other hand, such a variety can be obtained as a section of $\Gr(2,5)\subset \PP^9$
by a general hyperplane $\Lambda$ and two general Schubert subvarieties $\Sigma_1, \Sigma_2$ of codimension one in $\Gr(2,5)$
(see \cite{Todd1930}, \cite{Fujita-1986-1}).
Letting $\Sigma'$
be a general linear combination of $\Sigma_1$ and $\Sigma_2$, the section of $\Gr(2,5)\subset \PP^9$ 
by $\Lambda$ and $\Sigma'$ is smooth. Therefore, this section is
a del Pezzo fourfold $W=W_5\subset \PP^7$. By construction, $W$ contains $F$ and $\mathcal{P}$. 
Since $F\cdot \mathcal{P}=2$ in $W$, it follows that $F$ is of type \ref{case-g=9-a-26}, see (2) in Lemma \ref{lem:H4} 
and Lemma \ref{lem:enumerative}.

Since $U$ contains a unique
plane $\mathcal{P}$, and $F$ meets $\mathcal{P}$ just in two points and not along a conic,
$F$ satisfies the last condition of Proposition \ref{lemma-existence-surfaces}. Indeed, it is easily seen that
$U=W\cap\langle F\rangle$. If $\mathcal{T}$ is a plane, which meets $F$ along a conic, then $\mathcal{T}$ is contained in $U$. 
So, $\mathcal{T}=\mathcal{P}$ due to the uniqueness of $\mathcal{P}$. The latter equality leads to a contradiction, 
since $\mathcal{P}\cap F$ is not a conic. 
\end{proof}

To show
Proposition \ref{lemma-existence-surfaces} {\rm a)} we need to recall Proposition 4.11 in \cite{Prokhorov-Zaidenberg-2014}. 
It describes a construction (borrowed in \cite[Sect.\ 10]{Fujita-620281} and
\cite{Prokhorov1994}), which allows to recover the fourfold $W$ via a Sarkisov link starting with a certain 2-dimensional 
cubic scroll $S$ in $\PP^5$ contained in a smooth quadric $Q^4$.

\begin{proposition}\label{prop:PZ4.11}
Let as before $W=W_5\subset \PP^7 $ be a del Pezzo quintic fourfold, and let $l\subset W$ be a line, which is 
not contained 
in any plane in $W$, that is, $l\not\subset R$.
Then there is a commutative diagram
\begin{equation}\label{eq:diagram-Q4}
\vcenter{
\xymatrix{
&\hat D\ar[dl]&\hookrightarrow&\widehat W\ar[dr]^{\hat\varphi}\ar[dl]_{\hat\rho}&\hookleftarrow& \hat E\ar[dr]
\\
l&\hookrightarrow&W\ar@{-->}[rr]^{\hat\phi} &&\quad Q^4 &\hookleftarrow& S
}}
\end{equation}
where
\begin{enumerate}
\item
$\hat\rho: \widehat W\longrightarrow W$ is the blowup of $l$,
$\hat\phi: W\dashrightarrow \PP^5$ is the projection from $l$,
$Q^4=\hat\phi(\widehat W)\subset \PP^5$ is a smooth quadric, and
$\hat\varphi:\widehat W\longrightarrow Q^4$ is
the blowup of a cubic scroll $S\subset Q^4\subset \PP^5$ with exceptional divisor $\hat E$;

\item
the morphism $\hat\rho: \widehat W \longrightarrow W\subset \PP^7$ is defined by the linear system 
$|\hat\rho^* H - \hat D|$,
where $H\subset W$ is a hyperplane section and $\hat D =\hat\rho^{-1}(l)$
is the exceptional divisor of $\hat\rho$;

\item
$\hat\varphi(\hat D) = Q^4\cap \langle S\rangle$ is a quadric cone, where $\langle S\rangle\cong\PP^4$
is the linear span of $S$ in $\PP^5$;

\item
the image $\hat\rho(\hat E)\subset W$
is a hyperplane section of $W$ singular along $l$ and swept out by lines in $W$ meeting $l$;

\item
for a hyperplane section $\mathcal{L}$ of $Q^4$ we have on $\widehat W$
\begin{equation*} 
\hat\varphi^* \mathcal{L}\sim \hat D - \hat E\quad\mbox{and}\quad \hat\rho^*H\sim\hat D-2\hat E\sim \hat\varphi^* 
\mathcal{L}-\hat E\sim 2 \hat\varphi^* \mathcal{L}- \hat D\,.
\end{equation*}
\end{enumerate}
Conversely, given a pair $(Q^4,S)$, where $Q^4\subset \PP^5$ is a smooth quadric fourfold and $S\subset Q^4$is a cubic 
scroll in $ \PP^5$ such that the hyperplane section $Q^4\cap \langle S\rangle$ is a quadric cone, one can recover the 
quintic fourfold $W$ together with diagram \eqref{eq:diagram-Q4}) satisfying {\rm (i)-(v)}.
\end{proposition}

To construct surfaces $F\subset W$ as in Proposition \ref{lemma-existence-surfaces} {\rm a)} we use the following Lemmas 
\ref{lem:cubic scrolls-1}-\ref{lem:cubic scrolls-2}.

\begin{lemma}\label{lem:cubic scrolls-1}
Consider a quadric cone threefold $Q^3\subset\PP^4$ with a zero-dimensional vertex $P$, a smooth hyperplane section 
$Q^2=Q^3\cap \mathcal{H}$, where $Q^2\cong\PP^1\times\PP^1$, and a smooth
conic $C\subset Q^2$. Consider also a plane $\mathcal{T}\subset Q^3$, $\mathcal{T}\cong\PP^2$, and a general quadric $Q^{\bullet 3}
\subset\PP^4$ which contains $\mathcal{T}\cup C$. Then $Q^3\cap Q^{\bullet 3}=\mathcal{T}\cup S$, where 
$S\cong\FF_1$ is a smooth rational normal cubic scroll in $\PP^4$ passing through $P$ and $C$.
\end{lemma}

\begin{proof}
The exact sequence
\begin{equation*}
0\longrightarrow \mathcal{O}_{Q^3}(1)\longrightarrow \mathcal{O}_{Q^3}(2)
\longrightarrow \mathcal{O}_{Q^2}(2)\longrightarrow 0\,
\end{equation*}
yields the exact cohomology sequence
\begin{equation}\label{eq-cohom}
0\to H^0(\mathcal{O}_{Q^3}(1))\longrightarrow H^0(\mathcal{O}_{Q^3}(2))\stackrel{\psi}{\longrightarrow} H^0(\mathcal{O}_{Q^2}(2))\to 0\,.
\end{equation}
Let $l_1$ and $l_2$ be general horizontal and vertical generators of the quadric $Q^2$, and let $s\in H^0(\mathcal{O}_{Q^2}(2))$ be a 
section vanishing along the $(2,2)$-divisor $C+l_1+l_2$. By virtue of \eqref{eq-cohom} the affine subspace $\psi^{-1}(s)\subset H^0(\mathcal{O}_{Q^3}(2))$ 
has dimension $5$. It projects into a 5-dimensional family of divisors 
$D\in |\mathcal{O}_{Q^3}(2)|$ such that $D\cap Q^2= C+l_1+l_2 $. 
The plane $\mathcal{T}\subset Q^3$ is spanned by $l_1$ and $P$. It defines a 
2-dimensional subfamily $\mathcal{Q}$ of divisors $D$ containing 
$\mathcal{T}$ and such that $D\cap Q^2=C+l_1+l_2$.

Write $D=\mathcal{T}\cup S$, where $S$ is the residual cubic surface. Then $S\cap Q^2=C+l_2$. 
Suppose that $S$ is reducible: $S=\mathcal{T}_2\cup S'$, 
where $\mathcal{T}_2\cap Q^2=l_2$ and $S'\cap Q^2=C$. Then $D=\mathcal T_1\cup \mathcal T_2\cup S'$, where $\mathcal T_1=\mathcal T$, 
$\mathcal{T}_2={\rm span} (l_2,P)$ is a plane, and $S'$ is a hyperplane section of $Q^3$. 
Here $\mathcal T_1 \cup \mathcal T_2$ 
is uniquely determined by $l_1\cup l_2$, and $S'$ runs over a 1-parameter family. 

Since $\dim\mathcal{Q}=2$, one can conclude that a general divisor $D\in \mathcal Q$ has the form
$D=\mathcal{T}\cup S$, where $S\subset \PP^4$ is an irreducible cubic surface.

The cubic surface $S$ is linearly nondegenerate, because a hyperplane section of $Q^3$ is a quadric surface. 
Thus, $S$ is a linearly nondegenerate surface of minimal degree $3$ in $\PP^4$. Such a surface is either a cone over 
a twisted cubic $\Gamma\subset\PP^3$, or a rational normal scroll $S=S_{2,1}\cong\FF_1$ 
(see \cite[Ch.\ 4, Prop.\ on p.\ 525]{Griffiths-Harris-1978}).

If $S$ were a cone over $\Gamma\subset\PP^3$ with vertex $P'$, then the twisted cubic $\Gamma$ would 
be dominated by the conic $C$ under the projection from $P'$, which is impossible. Thus $F\cong\FF_1$ is smooth.

Finally, $P\in S$ since otherwise $S$ would be a Cartier divisor on $Q^3$ linearly proportional to a hyperplane section.
\end{proof}

\begin{lemma}\label{lem:cubic scrolls-2}
Let $Q^4\subset\PP^5$ be a smooth quadric. There exist two smooth cubic scrolls $S$ and 
$S'$ in $Q^4\subset \PP^5$ such that
\begin{itemize}
\item
$S\cong\FF_1\cong S'$;
\item
$S$ and $S'$ span hyperplanes $L$ and $L'$ in $\PP^5$, respectively, where $L\neq L'$;
\item 
$L\cap Q^4=Q^3$ and $L'\cap Q^4={Q'}^3$
are quadric cones with zero-dimensional vertices $P$ and $P'$, respectively, where $P\neq P'$;
\item 
the scheme theoretical intersection $C=S\cdot S'$ is a smooth conic.
\end{itemize}
\end{lemma}

\begin{proof}
A general pencil $(Q^3_{\lambda})$ of hyperplane sections of $Q^4$ contains exactly 6 degenerate members. 
Consider two of them, say, $Q^3=Q\cap T_{P}Q$ and ${Q'}^3=Q\cap T_{P'}Q$, where $P,P'\in Q$. Then $Q^3$ 
and ${Q'}^3$ are quadric cones with zero-dimensional vertices $P$ and $P'$, respectively. The base locus 
of the pencil $(Q^3_{\lambda})$ is the smooth quadric $Q^2=Q^3\cap {Q'}^3\cong\PP^1\times\PP^1$.
Applying Lemma \ref{lem:cubic scrolls-1} to $Q^3$ and ${Q'}^3$, the assertions follow.
\end{proof}

Using Lemma \ref{lem:cubic scrolls-2} and Proposition \ref{prop:PZ4.11} we proceed now with construction of surfaces $F$ 
as in Proposition \ref{lemma-existence-surfaces} a).

\begin{construction}\label{sit} {\rm Consider the smooth cubic scrolls $S$ and $S'$ in $\PP^5$ as in Lemma 
\ref{lem:cubic scrolls-2}. The embedding $\FF_1\stackrel{\cong}{\longrightarrow} S'{\hookrightarrow}\PP^4$ 
is given by the linear system $|\sigma+2f|$ on $\FF_1$, where $\sigma$ is the exceptional section of 
$\FF_1\to\PP^1$ and $f$ is a fiber. On $S'$ we have $C=S\cdot S'\sim\sigma+f$, where the images of 
$\sigma$ and $f$ on $S'$ are denoted by the same letters.

In what follows we employ the notation of Proposition \ref{prop:PZ4.11}. Let $\hat S'$ be the proper transform of 
$S'$ in $\widehat W$ (see diagram \eqref{eq:diagram-Q4}). Then, clearly, $\hat S'\cong S'\cong \FF_1$. 
By Proposition \ref{prop:PZ4.11}(v), the morphism $\hat\rho: \widehat W\to W\subset\PP^7$ is defined by 
the linear system $|\hat\rho^*H|=|2\hat\varphi^*\mathcal{L}-\hat D|$, where $\mathcal{L}$ is a hyperplane section 
of $Q^4\subset\PP^5$ and $\hat D=\hat\varphi^*(S)$ is the exceptional divisor of $\hat\varphi$. Identifying $S'$ 
with $\tilde S'$ one can write
\begin{equation}\label{eq:morphism} 
(2\varphi^*\mathcal{L}-\hat D)|_{\hat S'}
=2\mathcal{L}|_{S'}-S|_{S'}
\sim 2(\sigma+2f)-C
\sim \sigma+3f\,.
\end{equation}
We let $F=\hat\rho (\hat S')\subset W$. Since $\hat S'\not\subset \hat D$, the map 
$\hat\rho|_{\hat S'}: \hat S' \to F$ is a birational morphism, and the surface $F$ is a quintic scroll.}
\end{construction}

\begin{remark}\label{rem:l}
Since $S'\cap \langle S\rangle =C+f_0$, where $f_0$ is a fiber of $S'$, we have $\hat S'\cap \hat D\supset \hat f_0$. 
Therefore, $\hat \rho(\hat f_0)=l\subset F$ (because $l=\hat\rho(\hat D)$ and $F=\hat\rho(\hat S')$).
Moreover, $l$ is a ruling of $F$. 
\end{remark}

\begin{lemma}\label{lem:quintic scroll} The morphism $\hat\rho|_{\hat S'}: \hat S'\to F$
is an isomorphism onto a smooth rational normal quintic scroll $F\supset l$ contained in a hyperplane in
$\PP^7$.
\end{lemma}

\begin{proof} It suffices to show that the morphism
$\hat\rho|_{\hat S'}: \hat S'\to \PP^6\subset\PP^7$ is given by the (very ample)
{\em complete} linear system $|\sigma+3f|$ on $\hat S'\cong\FF_1$ (cf.\ \eqref{eq:morphism}), or, in other words, that the induced morphism $\FF_1\to F$ is an isomorphism, see 
\cite[Ch.\ 4, p.\ 523]{Griffiths-Harris-1978} or \cite{Harris1992}.

Suppose to the
contrary that $\langle F\rangle\cong\PP^5$, that is, $F$ is cut out in $W$ by two hyperplanes. 
Then the quintic scroll $F$ 
cannot be normal. Indeed, for a general hyperplane section $\gamma$ on $F$ we have by adjunction 
$\omega_\gamma=(K_W+3H)|_\gamma\sim 0$. Hence the arithmetic genus of $\gamma$ equals 1. The genus of 
the proper transform of $\gamma$ on the normalization of $F$ equals 0, hence $\gamma$ is a rational 
curve with one double point. 
Such double points of hyperplane sections of $F$ fill in a line in $F$, and $F$ is singular along this line. In particular, $F$ is not normal.
This leads to the following claim. 

\begin{claim}\label{claim1}
If $\langle F\rangle\cong\PP^5$ then $\Sing F=l$ is a ruling of $F$. 
\end{claim}

\begin{proof}
We know that $l\subset F$ is a ruling, see Remark \ref{rem:l}.
Since $\hat W\to W$ is an isomorphism over $W\setminus l$, its restriction $\hat S'\to F$ is an isomorphism over $F\setminus l$.
Since $F$ is not normal, the claim follows. 
\end{proof}

On the other hand, we have

\begin{claim}\label{claim2}
Let as before $\langle F\rangle\cong\PP^5$, and let $\nu:\FF_1\to F$ be 
the normalization. Then on $\FF_1$ we have $K_{\FF_1}\sim \nu^*\omega_F-B$, 
where $B\sim \sigma$ is an effective divisor supported by the proper transform in $\FF_1$ 
of the non-normal locus of $F$. 
\end{claim}

\begin{proof} 
Under our assumption, $F$ is a complete intersection in a smooth variety $W$. 
Hence $F$ is Cohen-Macaulay, and so, the standard formula $K_{\FF_1}\sim \nu^*\omega_F-B$ 
holds with $B$ supported by the proper transform in $\FF_1$ of the non-normal locus of $F$. 
Using this formula and adjunction one gets on $\FF_1$:
\begin{eqnarray*} 
B\sim \nu^*\omega_F-K_{\FF_1}& \sim & (K_W+2H)|_F+(2\sigma+3f) \sim -H|_F+(2\sigma+3f)
\\ 
& \sim & -(\sigma+3f)+(2\sigma+3f) \quad\,\, \sim \sigma\,,
\end{eqnarray*}
as stated. 
\end{proof}

Due to Claim \ref{claim1} we have supp$(B)=f$, and so, $B\cdot f=0$. This yields a contradiction, 
since by Claim \ref{claim2}, $B\cdot f=\sigma\cdot f=1$ on $\FF_1$. 
\end{proof}

\begin{lemma}\label{lem:no plane conic}
None of the planes in $W$ meets the quintic scroll 
$F\subset W$ along a \textup(possibly, degenerate\textup) conic.
\end{lemma}

\begin{proof} 
Recall that $R$ stands for the hyperplane section of $W$ swept out by 
the 1-parameter family of planes $(\Pi_t)$ in $W$. It is singular along the plane $\Xi$, 
see Proposition \ref{Proposition-2.2.}(iv). Since $l\subset F$ and $l\not\subset R$, we 
have $F\not\subset R$ and $l\cap\Xi=\emptyset$, see Proposition \ref{Proposition-2.2.}(v). 

Suppose to the contrary that $F$ meets a plane $\mathcal{P}\subset W$ along a conic, say, $\eta$. 

\begin{claim*}
The conic $\eta$ coincides with the exceptional section $\sigma_F$ 
of the scroll $F\cong\FF_1$.
\end{claim*}

\begin{proof} 
Suppose that the conic $\eta$ is degenerate. Since any two lines on 
$F$ are disjoint, $\eta\subset P$ cannot be a bouquet of two distinct lines. Hence $\eta$ is a double line $2f$.

For any line $f'\neq f$ in $F$ there exists an automorphism $\alpha\in\Aut F\cong\Aut\FF_1$ 
such that $\alpha(f)=f'$. Since the embedding 
\begin{equation*}
\FF_1\stackrel{\cong}{\longrightarrow} F\hookrightarrow \PP^6\subset\PP^7
\end{equation*} 
is given by an $(\Aut F)$-invariant linear system $|\sigma+ 3f|$, $\alpha$ can be extended to an 
automorphism $\bar\alpha\in\Aut\PP^7$, which leaves $\langle F\rangle\cong\PP^6$ invariant. Hence 
there exists a second plane $\mathcal{P}'=\bar\alpha(\mathcal{P})$, which meets $F$ along a double 
line $2f'$ (this plane $\mathcal{P}'$ does not need to be contained in $W$).
\footnote{Alternatively, the further proof can proceed as follows. The plane $\mathcal{P}'$ is tangent 
to $F$ along the ruling $f'$. Thus, the Gauss map of $F$ is degenerate, and so, $F$ is a developable 
surface. Such a surface, which is not a plane, is a cone or the tangential developable of a curve, 
see, e.g., \cite{Cayley1864} or more general results in \cite[(2.29)]{Griffiths-Harris-1978}, 
\cite[Cor.\ 5]{Zak1987}, or \cite[\S 2.3.3]{Fischer-Piontkowski2001}. Hence $F$ cannot 
be smooth, a contradiction. Our argument in the text is more elementary.}

The planes $\mathcal{P}$ and $\mathcal{P}'$ span a subspace $\mathcal{N}\subset\PP^7$ with $\dim\mathcal{N} \le 5$.
Thus, there exists a hyperplane $\mathcal{M}\supset \mathcal{N}$ in $\PP^7$ different from $\langle F\rangle$. We have 
$\mathcal{M}\cdot F=2f+2f'+f''$, where $f''\subset F$ is an extra line. However, this divisor $\mathcal{M}\cdot F$ on $F$ is not ample, 
which is a contradiction.

Thus, the conic $\eta=F\cap\mathcal{P}$ is smooth. Since the image $\sigma$ of the exceptional section $\sigma_{\hat S'}\subset\hat S'$ 
is a unique smooth conic in
the quintic scroll $F\cong\FF_1$, we obtain that $\eta=\sigma_F$. 
\end{proof}

The line $l\subset F$ meets the section $\sigma=\sigma_F$ in a point $p\in\sigma$. Hence it meets also the plane $\mathcal{P}$ in $p$. 
The projection $\hat\phi: W\dashrightarrow\PP^5$ with center $l$ sends $\sigma_F$ to the exceptional section $\sigma_{S'}\subset S'$, and 
$\mathcal{P}$ to a line on $S'\cong\FF_1$, which should coincide with $\sigma_{S'}$. Recall that by our construction 
$S\cap S'=C\sim\sigma_{S'}+f_{S'}$ is a smooth conic on $S'$. Since
$\sigma_{S'}\cap C=\emptyset$, the exceptional divisor $\hat E\subset\widehat W$ does not meet the section $\sigma_{\hat S'}$ of the scroll 
$\hat S'\subset\widehat W$. Thus $\hat\varphi:\widehat W\to Q^4$ is an isomorphism near $\sigma_{\hat S'}$. 

On the other hand, let $\hat{\mathcal{P}}$ be the proper transform of $\mathcal{P}$ in $\widehat W$. Then $\hat{\mathcal{P}}\to\mathcal{P}$ 
is the blowup of the point $p=\mathcal{P}\cap l$, and $\hat{\mathcal{P}}\cap\hat S'\supset\sigma_{\hat S'}$. Thus the image 
$\hat\varphi(\hat{\mathcal{P}})\subset Q^4$ should be a surface, and not a line. 
This yields as well a contradiction. 
\end{proof}

Examples show that the last assumption in Theorem \ref{theorem-1} cannot be omitted. 
Without this assumption one arrives at a singular fourfold $V$ in diagram \eqref{eq:diagram}, 
or else $\varphi$ is the blowup of a singular surface. According to Proposition \ref{lemma-existence-surfaces}, 
this does not happen for our choice of $F$. 

\section{Proof of Theorem \ref{theorem-1}.}\label{sec:proof-of-theorem-1}

Let us start with the following well known lemmas.

\begin{lemma}\label{lem:intersec-quadrics}
Any surface $F$ as in Theorem \textup{\textup{\ref{theorem-1}}} is a scheme theoretical intersection of quadrics.
\end{lemma}
\begin{proof} In case \ref{case-g=10} the assertion follows from
\cite[Thm.\ 8.4.1]{Dolgachev-ClassicalAlgGeom}, and in case \ref{case-g=9} from \cite[Lect.\ 9, Exs.\ 9.10--9.11]{Harris1992}.
\end{proof}

The next well known lemma is immediate.

\begin{lemma}\label{lem:morphism}
Let a smooth surface $F\subset \PP^n$, $n\ge 4$, be a scheme theoretical 
intersection of quadrics. Let $\tilde \PP^n\to \PP^n$ be the blowup of $F$ 
with exceptional divisor $T$. Then the linear system $|2H^* - T |$ defines a 
morphism $\tilde \PP^n\to \PP^N$, which contracts the proper transform of 
any $2$-secant line of $F$.
\end{lemma}

\begin{sit}\label{sit:pf-2.1} In what follows we keep the notation as in 
Theorem \textup{\ref{theorem-1}}. In particular, we let $g=10$ in case \ref{case-g=10} 
and $g=9$ in case \ref{case-g=9}. 

A surface $F\subset W$ as in Theorem \textup{\ref{theorem-1}} is contained in a unique 
hyperplane section $E=\langle F\rangle\cap W$ of $W$, see Remark \ref{lem:unique-hyperplane}. We let 
\begin{itemize}
\item 
$\rho:\widetilde W\longrightarrow W$ be the blowup of $F$
with exceptional divisor $D$, 
\item 
$\tilde E\subset\widetilde W$ be the proper transform of $E$, 
\item 
$H\subset W$ be a general hyperplane section, and 
\item 
$H^*=\rho^*H\in {\rm Div}\,W$. 
\end{itemize}
Clearly, one has $\rk\Pic\widetilde W=2$ and $\tilde E\sim H^*-D$ on $\widetilde W$. 
\end{sit}

\begin{lemma}\label{lem:Fano}
The variety $\widetilde W$ is a smooth Fano fourfold. 
\end{lemma}

\begin{proof}
We have
\begin{equation*}
-K_{\widetilde W}=3H^*-D=2H^*-D+H^*\,,
\end{equation*} 
where both $2H^*-D$ and $H^*$ are nef, because the linear systems
$|2H^*-D|$ and $|H^*|$ are free. Since $\rk\Pic\widetilde W=2$ and the 
nef divisors $2H^*-D$ and $H^*$ are not proportional,
their sum is an ample divisor by the Kleiman ampleness criterion. 
\end{proof}

The nef and non-ample linear systems $|H^*|$ and $|2H^*-D|$ on $\widetilde W$ define 
the two extremal Mori contractions on $\widetilde W$. The first one is $\rho:\widetilde W\to W$; 
the second one $\varphi:\widetilde W\to V$ makes the subject of our following studies. We need the next lemma. 

\begin{lemma}\label{lemma-equation-intersection}
On $\widetilde W$ one has $({H^*})^4=5$,\quad $(H^*)^3\cdot D=0$,
\begin{equation*}
(H^*)^2\cdot D^2=
\begin{cases}
-5
\\
-6
\end{cases},
\
H^*\cdot D^3=\begin{cases}
-8
\\
-12
\end{cases},
\
D^4=\begin{cases}
-6&\text{in case \ref{case-g=10}}
\\
-16&\text {in case \ref{case-g=9}\,.}
\end{cases}
\end{equation*}
\end{lemma}

\begin{proof} The lemma follows easily from the equalities (see \cite[Lem.\ 1.4]{Prokhorov-Zaidenberg-2014})
\begin{equation*}
(H^*)^2\cdot D^2=-F\cdot H^2\,,
\end{equation*}
\begin{equation*}
H^*\cdot D^3=-H|_F\cdot K_F-3H\cdot H\cdot F\,,
\end{equation*} 
and
\begin{equation*}
D^4= c_2(W) \cdot F+K_W|_F\cdot K_F-c_2(F)-K_W^2\cdot F\,.
\end{equation*}
\end{proof}

\begin{lemma}\label{lem:5.6} Let $U$ be a Mukai fourfold of genus 
$g(U)\ge 4$ with at worst terminal Gorenstein singularities and with $\rk \Pic U=1$. 
Assume that the linear system $|-\frac12 K_U|$ is base point free.
Then the divisor $-\frac 12 K_U$ is very ample and defines an embedding $U\hookrightarrow\PP^{g+2}$. 
\end{lemma}

\begin{proof}
This follows from the corresponding result in the three-dimensional case, see \cite[Prop.\ 1]{Mukai-1989}, \cite{Iskovskikh-Prokhorov-1999}, and
\cite{Przhiyalkovskij-Cheltsov-Shramov-2005en}, by recursion on the dimension, likewise this is done in \cite[Lem.\ (2.8)]{Iskovskih1977a}. 
\end{proof}

\begin{sit}\label{sit:Stein factorization} Using Lemma \ref{lemma-equation-intersection} we obtain 
\begin{equation}\label{eq:4th power}
\deg V=(2H^*-D)^4=2g-2=\begin{cases} 18 & \mbox{in case \ref{case-g=10}}\\
16& \mbox{in case \ref{case-g=9}}\,
\end{cases}\end{equation}
and
\begin{equation}\label{eq:LE=0} 
\tilde E\cdot (2H^*-D)^3=(H^*-D)\cdot (2H^*-D)^3=0\,.
\end{equation}
Therefore,
the linear system $|2H^*-D|$ defines a generically finite morphism 
\begin{equation*}
\Phi_{|2H^*-D|}:\widetilde W\to V\subset\PP^{g+2}\,
\end{equation*}
onto a fourfold $V$, where $\Phi_{|2H^*-D|}$ contracts the divisor $\tilde E\sim H^*-D$.
Consider the Stein factorization
\begin{equation*}
\Phi_{|2H^*-D|}:\widetilde W\stackrel{\varphi}{\longrightarrow} U\to V\subset\PP^{g+2}\,.
\end{equation*}
Here $\varphi$ is a divisorial Mori contraction, and $\Pic U=\ZZ\cdot L$, where $L$ is an ample 
Cartier divisor with $\varphi^*L=2H^*-D$.
Once again, the exceptional divisor of $\varphi$ is $\tilde E\sim H^*-D$.
Hence $D\sim\varphi^* L-2E$. 
\end{sit}

\begin{lemma}\label{lem:U-Mukai}
The variety $U$ as in \textup{\ref{sit:Stein factorization}} is a Mukai fourfold with at worst terminal 
Gorenstein singularities and $\rk\Pic U=1$. 
\end{lemma}

\begin{proof}
Since $\varphi$ is a divisorial Mori contraction,
$U$ has at worst terminal singularities. We have $\rk\Pic U=1$ because $\rk \Pic \widetilde W=2$. 
Since 
\begin{equation*}
-K_{\tilde W}= 3H^*-D= 2(2H^*-D) -\tilde E\,
\end{equation*} 
we also have
$ -K_U = 2L$. Hence $-K_U$ is an ample Cartier divisor divisible by $2$ in $\Pic U$.
So $U$ is a Mukai fourfold. 
\end{proof}

\begin{convention}
\label{nota:U-V} 
The morphism $U\to V\subset \PP^{g+2}$ is 
given by the linear system $|L|=|-\frac12 K_U|$.
As follows from Lemma \ref{lem:5.6}, this is an isomorphism.
In the sequel we identify $V$ with $U$ and $\Phi_{|2H^*-D|}$ with $\varphi$.
\end{convention}

\begin{lemma}\label{cor:5.4}
For the image $V= \varphi (\widetilde W)\subset \PP^{g+2}$
the following hold.
\begin{enumerate}
\item
The morphism
$\varphi: \widetilde W\to V$ is birational and $\deg V=2g-2$;
\item
the morphism $\varphi$ contracts the divisor $\tilde E$
to an irreducible surface $S\subset V$;
\item
$\deg S= g-7=\begin{cases} 3 & \mbox{in case}\quad \ref{case-g=10}\\
2 & \mbox{in case}\quad \ref{case-g=9}\,;
\end{cases}$
\item
$S$ can have only isolated singularities.
\end{enumerate}
\end{lemma}

\begin{proof} 
Upon convention \ref{nota:U-V}, $\varphi$ is birational. By \eqref{eq:4th power}
we have $\deg V=2g-2$. By virtue of \eqref{eq:LE=0}, $\tilde E$ is the exceptional 
divisor of $\varphi$.
Using Lemma \ref{lemma-equation-intersection} we deduce the equalities
\begin{eqnarray*} 
(2H^*-D)^2\cdot \tilde E^2= (2H^*-D)^2(H^*-D)^2=\begin{cases}
-3 & \mbox{in case}\quad \ref{case-g=10}\\
-2 & \mbox{in case}\quad \ref{case-g=9}\,.
\end{cases}
\end{eqnarray*}
Since the latter number is nonzero, $S$ is a surface of degree 
\[
\deg S=- (2H^*-D)^2\cdot \tilde E^2 
\]
satisfying (iii). 

Since $\rk\Pic\widetilde W=2$,
the exceptional locus of $\varphi$ coincides with $\tilde E$, and $\tilde E$ is a prime divisor. 
Therefore, $\varphi$ has at most a finite number of 2-dimensional fibers. By the Andreatta--Wisniewski 
Theorem (\cite{AndreattaWisniewski1998}) $S$ has at most isolated singularities. 
\end{proof}

\begin{corollary}\label{lem:normal} 
The surface $S$ is normal.
\end{corollary}

\begin{proof}
The assertion is certainly true if $\deg S=2$. If $\deg S=3$ and the cubic surface $S\subset\PP^4$ is not normal, 
then $S$ is contained in a 3-dimensional subspace and the singular locus of $S$ is 1-dimensional, which 
contradicts (iv).
\end{proof}

\begin{lemma}\label{lem:replacement of 2.6 and 2.7}
In the notation of \textup{\ref{sit:Stein factorization}}
the morphism $\varphi:\widetilde W\to V$ is the blowup of the surface $S$, where both $S$ and $V$ are smooth. 
\end{lemma}

\begin{proof}
If to the contrary $S$ or $V$ were singular, then by \cite[Thm.\ 2.3]{Ando1985} the extremal 
$K_{\widetilde W}$-negative contraction $\varphi: \widetilde W\to V$ would have a 2-dimensional 
fiber, say, $\widetilde Y\subset\widetilde W$. 
Since $S$ is normal (see Corollary \ref{lem:normal}), by the main theorem and Prop.\ 4.11 in 
\cite{AndreattaWisniewski1998} one has
$\widetilde Y\cong \PP^2$ and 
\begin{equation*}
(3H^*-D)|_{\widetilde Y}=-K_{\widetilde W}|_{\tilde Y}=\mathcal{O}_{\PP^2}(1)\,.
\end{equation*} 
Since $\tilde Y$ is contracted to a point under $\varphi$, we have $(2H^*-D)|_{\widetilde Y}\sim 0$. Thus
$H^*|_{\widetilde Y}=\mathcal{O}_{\PP^2}(1)$ and $D|_{\widetilde Y}=\mathcal{O}_{\PP^2}(2)$.

It follows that the image $Y=\rho(\widetilde Y)\subset W$, where $\rho=\Phi_{|H^*|}$, 
is a plane, $Y\neq F$, and $Y\cap F\cong \widetilde Y\cap D$ is a conic in $Y\cong\PP^2$. 
However, the latter contradicts our assumption in Theorem \textup{\ref{theorem-1}} that $F$ 
does not meet any plane in $W$ along a conic.

Therefore, $\varphi$ has no 2-dimensional fiber. Hence the surface $S$ and the fourfold $V$ 
are smooth, and $\varphi$ is the blowup of $S$ by \cite[Thm.\ 2.3]{Ando1985}. 
\end{proof}

\begin{corollary}\label{cor:new}
The surface $S\subset V\subset\PP^{g+2}$ is a smooth normal cubic scroll in case \ref{case-g=10}
and a smooth quadric in case \ref{case-g=9}. 
\end{corollary}

\begin{proof}
By Lemmas \ref{cor:5.4}(iii) and \ref{lem:replacement of 2.6 and 2.7}, $S$ is a 
smooth surface of degree 3 in case \ref{case-g=10} and of degree 2 in case \ref{case-g=9}. 
It remains to show that in case \ref{case-g=10}, $S$ is a normal scroll in $\PP^4$ and not a 
smooth cubic surface in $\PP^3$. Using \eqref{equation-intersections} and Lemma 
\ref{lemma-equation-intersection} one can compute
\begin{equation*}
L^*\cdot \tilde E^3=(2H^*-D)\cdot (H^*-D)^3= -1.
\end{equation*}
On the other hand, 
\begin{equation*}
L^*\cdot \tilde E^3= -L|_{S}\cdot K_S+ K_V\cdot L\cdot S
\end{equation*}
(see e.g. \cite[Lem.\ 1.4]{Prokhorov-Zaidenberg-2014}), and so, due to \ref{nota:U-V},
\begin{equation*}
L|_{S}\cdot K_S= -L^*\cdot \tilde E^3-2 L^2\cdot S=1-6=-5.
\end{equation*}
If $\dim \langle S\rangle <4$, then $S$ is a cubic surface in $\PP^3$
and we have $L|_{S}\cdot K_S=-K_S^2=-3$, a contradiction.
Therefore, $\dim \langle S\rangle=4$, and so, $S\subset \PP^4$ is a linearly nondegenerate 
surface of degree $3$, i.e., a normal cubic scroll.
\end{proof}

\begin{lemma}\label{lem:U=V} 
Under the setting as before, the following hold. 
\begin{itemize}
\item
$\varphi(D)$ is a hyperplane section of $V$ singular along $S=\varphi(\tilde E)$, 
\item
there is an isomorphism $V\setminus\varphi(D)\cong W\setminus\rho(\tilde E)$.
\end{itemize}
\end{lemma}

\begin{proof}
We have $D\sim\varphi^*L-2\tilde E$ in $\widetilde W$ and $S\subset\varphi(D)$, 
because any fiber $\varphi^{-1}(s)$, $s\in S$, meets $D$. Thus, $\varphi(D)\sim L$ is 
a hyperplane section of $V\subset\PP^{g+2}$ singular along $S=\varphi(\tilde E)$. 

Finally, since $F\subset E=\rho(\tilde E)$ we have isomorphisms 
\begin{equation*}
W\setminus\rho(\tilde E)\cong \widetilde W\setminus (\tilde E\cup D)\cong 
V\setminus (S\cup\varphi(D))=V\setminus \varphi(D)\,.
\end{equation*}
\end{proof}

The following corollary is immediate from \eqref{eq:4th power}
and Lemma \ref{lem:U-Mukai}. It ends the proof of Theorem \ref{theorem-1}.

\begin{corollary}\label{main-corollary}
Under the assumptions of Theorem \textup{\ref{theorem-1}}, 
\begin{itemize}
\item 
in case \ref{case-g=10}
$V\cong V_{18}\subset\PP^{12}$ is a smooth Mukai fourfold of genus $g=10$, and
\item 
in case \ref{case-g=9} $V\cong V_{16}\subset\PP^{11}$ is a smooth Mukai fourfold of genus $g=9$.
\end{itemize}
\end{corollary}

\section{Concluding remarks.}
\subsection{Cylindricity in families}
Our Theorem \ref{main-theorem} and the results in \cite{Prokhorov-Zaidenberg-2014}
show that for any $g\ge 7$, in the family of all 
Mukai fourfolds of genus $g$ there exist subfamilies of cylindrical such fourfolds.
The question about cylindricity of all the Mukai fourfolds of genus $g\ge 7$ remains open,
and as well the question about cylindricity of Mukai fourfolds
of lower genera is. We expect that the answers to both questions are negative in general.
However, at the moment we do not dispose suitable tools to prove this.

\subsection{Rationality questions}
The question about cylindricity is ultimately related to the rationality problem.
For instance, in dimension $3$ cylindricity of a Fano variety implies its rationality.
Note that for any $g=5,\ldots, 8$ there exist 
rational Mukai fourfolds $V=V_{2g-2}\subset \PP^{g+2}$ of genus $g$. We also have the following fact. 

\begin{proposition}\label{prop:rationality}
Any Mukai fourfold $V=V_{2g-2}\subset \PP^{g+2}$ of genus $g\in\{7, 9, 10\}$ is rational. 
\end{proposition}

\begin{proof} By Shokurov's theorem (\cite{Shokurov-1979}) applied to a hyperplane section, 
there exists a line $\lambda$ on $V$.
By an easy parameter count (see \cite[Lem.\ 2.4]{Prokhorov-Zaidenberg-2014}) a general hyperplane
section of $V$ passing through $\lambda$ is smooth. Hence one can take a pencil $\mathcal{H}$
of hyperplane sections of $V$ passing through $\lambda$ whose general member $U=H_{2g-2}\in \mathcal H$ is
a smooth anticanonically embedded
Fano threefold of genus $g$ with $\Pic U=\ZZ\cdot K_U$. 
Blowing up the base locus of $\mathcal H$ yields a family $\mathfrak V\to \PP^1$, 
whose fibers are the members of $\mathcal H$ and the total space $\mathfrak V$ is birational to $V$. 

Consider the generic fiber 
$X=\mathfrak V\times \operatorname{Spec}\,\CC(\PP^1)$, where
$\PP^1$ is the parameter space of 
the pencil $\mathcal H$. 
As before, $X$ is a Fano threefold of genus $g$ over the non-closed field 
$\CC(\PP^1)$ with $\Pic X=\ZZ\cdot K_X$.
It suffices to show the $\CC(\PP^1)$-rationality of $X$. 

By construction, the line $ \lambda\subset V$ gives a
line $\Lambda \subset X$ defined over $\CC(\PP^1)$.
Then we can apply the Fano-Iskovskikh double projection $\Psi: X \dashrightarrow Y$ 
from $\Lambda$, see \cite{Iskovskikh-Prokhorov-1999}. For $g=9$ ($g=10$, respectively)
the map $\Psi$ is birational and $Y$ is a form of $\PP^3$, i.e., a Brauer-Severi scheme, 
(a smooth quadric $Q\subset \PP^4$, respectively). 
Since $\CC(\PP^1)$ is a $c_1$-field, by Tsen's theorem, $Y$ is 
$\CC(\PP^1)$-rational, and so, $X$ is as well. In the case $g=7$ we have $Y\cong \PP^1$ 
and $\Psi$ is a birational map to a del Pezzo fibration of degree $5$. 
Thus, the original variety $V$ has a birational structure of a del Pezzo fibration 
of degree $5$ over a surface. Then $V$ is rational by the Enriques--Manin--Swinnerton-Dyer
theorem (see, e.g., 
\cite{Shepherd-Barron-1992}).
\end{proof}

We do not know whether the rationality as in Proposition \ref{prop:rationality} holds also for the Mukai fourfolds $V_{2g-2}$ of genera $g=5,6,8$.

\subsection{Compactifications of $\CC^4$}
The Hirzebruch problem about compactifications of the affine space $\mathbb A^n$ (\cite{Hirzebruch-1954})
is also closely related to our cylindricity problem. One can ask the following natural question:

\begin{quote}
Which Mukai fourfolds can serve as compactifications of $\mathbb A^4$?
\end{quote}

We hope that the corresponding examples can be constructed via Sarkisov links,
likewise this is done in the present paper for cylindricity. 
For the del Pezzo fourfolds, a similar problem was completely solved in \cite{Prokhorov1994}.

\end{document}